\documentclass[12pt]{article}
\usepackage{amssymb, amsmath, std, Srcltx}
\usepackage{graphicx}
\usepackage{xcolor}

\newcommand\newsection[1]{\bigskip\bigskip\refstepcounter{section}
	\noindent{\large\bf\thesection . #1}}
\newcommand\newsubsection[1]{\medskip\refstepcounter{subsection}
	{\noindent\bf #1 \thesubsection .\ }}

\newenvironment{theorem}{\newsubsection{Theorem}\sl}{}
\newenvironment{corollary}{\newsubsection{Corollary}\sl}{}
\newenvironment{lemma}{\newsubsection{Lemma}\sl}{}

\newenvironment{definition}{\newsubsection{Definition}}{}

\newenvironment{proof}{\medskip\noindent{\bf Proof.\ }}{\mbox{$\square$}}

	\newcommand\bl{\begin{lemma}}
		\newcommand\el{\end{lemma}}
	\newcommand\bt{\begin{theorem}}
		\newcommand\et{\end{theorem}}
	\newcommand\be{\begin{equation}}
	\newcommand\ee{\end{equation}}
	\newcommand\bea{\begin{eqnarray}}
	\newcommand\eea{\end{eqnarray}}

	\newcommand\integers{{\mathbb Z}}

\begin{document}

\centerline{\large \bf The girth, odd girth, distance function,}

\medskip

\centerline{\large \bf  and diameter of generalized Johnson graphs}

\bigskip
\bigskip

\centerline{John S. Caughman\footnote[1]{Corresponding author:  caughman@pdx.edu.\hfill AMS 2010 Subject Classification: 05C12}, Ari J. Herman, Taiyo S. Terada}

\medskip

\centerline{\emph{Department of Mathematics \& Statistics}}

\centerline{\emph{Portland State University, Portland, OR, USA}}

\bigskip

\begin{abstract}
For any non-negative integers $v>k>i$, the generalized Johnson graph, $J(v,k,i)$, is the undirected simple graph whose vertices are the $k$-subsets of a $v$-set, and where any two vertices $A$ and $B$ are adjacent whenever $|A\cap B|=i$.  In this article, we derive formulas for the girth, odd girth, distance function, and diameter of $J(v,k,i)$. In particular, let $X = J(v,k,i)$.  Assume $v \geq 2k$ and $(v,k,i) \not= (2k,k,0)$. For convenience, abbreviate $\Delta=v-2k+2i$. We prove the following.
\begin{enumerate}
    \item The girth of $X$ is given by
$$g(X)= \left\{
\begin{array}{ll}
3 & \hbox{if \, } v\geq 3(k-i); \\
4 & \hbox{if \, } v<3(k-i) \, \mbox{ and } \, (v,k,i)\not= (2k+1,k,0); \\
5 & \hbox{if \, } (v,k,i)=(5,2,0); \\
6 & \hbox{if \,} (v,k,i)=(2k+1,k,0) \, \mbox{ and } \, k>2.
\end{array}
\right.$$
\vspace{-1em}
    \item The odd girth of $X$ is given by
 \begin{equation*} 
 og(X)=2\left\lceil\frac{k-i}{\Delta}\right\rceil+1.
 \end{equation*}
\vspace{-1em}
    \item Let $A$ and $B$ be vertices of $X$ and let $x=|A\cap B|$.  Then
\begin{equation*} 
\text{dist}(A,B)= \left\{
\begin{array}{ll}
3 & \hbox{if \, } x<\min\{i,k-\Delta\}; \\
\lceil\frac{k-x}{k-i}\rceil & \hbox{if \, } k-\Delta\leq x< i; \\
\min\{2\lceil \frac{k-x}{\Delta} \rceil,2\lceil \frac{x-i}{\Delta} \rceil +1\} & \hbox{if \, } x \geq i. 
\end{array}
\right.
\end{equation*}
\vspace{-1em}
    \item The diameter of $X$ is given by
$$ \text{diam}(X)= \left\{
\begin{array}{ll}
\lceil \frac{k-i-1}{\Delta} \rceil+1 & \hbox{if \, } v<3(k-i)-1 \text{ or } i=0; \\
3 & \hbox{if \, } 3(k-i)-1\leq v <3k-2i \text{ and } i\neq 0;\\
\lceil\frac{k}{k-i}\rceil & \hbox{if \, } v \geq 3k-2i \text{ and } i\neq 0. \\
\end{array}
\right.$$
\end{enumerate}
\end{abstract}

\bigskip 

\centerline{\newsection{Introduction}}

Fix any non-negative integers $v>k>i$. The {\em generalized Johnson graph}, $X=J(v,k,i)$, is the undirected simple graph whose vertices are the $k$-subsets of a $v$-set, and where any two vertices $A$ and $B$ are adjacent whenever \mbox{$|A\cap B|=i$}. Generalized Johnson graphs were introduced by Chen and Lih in \cite{chen2} and have also been studied under the name {\em uniform subset graphs}.  Special cases include the Kneser graphs $J(v,k,0)$, the odd graphs $J(2k+1,k,0)$, and the Johnson graphs $J(v,k,k-1)$.  

In what follows, we derive formulas for the girth, odd girth, distance function, and diameter of $J(v,k,i)$. Some special cases have been previously determined. 
 The classical Johnson graph $J(v,k,k-1)$ is well known to have diameter $\min\{k,v-k\}$, and formulas for the distance and diameter of Kneser graphs were proved in \cite{valencia}.  The diameter of $J(v,k,i)$ was studied in \cite{chen}; however, a formula there gives incorrect values when $i>\frac{2}{3}k$, an important case that includes the classical Johnson graphs. In this paper, we extend and correct such expressions. Regarding odd girth, the Kneser graphs are known to have odd girth $2\lceil \frac{k}{v-2k} \rceil +1$, as proved in \cite{poljak}, which meets the bound given in \cite[p.146]{godsil}, and which simplifies to $2k+1$ in the case of the odd graphs. The Johnson graphs are distance-regular (see \cite{biggs} or \cite{godsil}) but not triangle-free when $v>2$, so they have girth (and odd girth) 3.  

We note that it is possible to extend the definition of $X=J(v,k,i)$ to include cases allowed by the weaker inequalities $v\geq k \geq i$.  However, $X$ is an empty graph when $k = i$ or $v=k$.  If $v = 2k$ and $i = 0$, then $X$ is 
isomorphic to the disjoint union of ${2k}\choose{k}$$/2$ copies of $K_2$. Furthermore, by taking complements, the graphs $J(v,k,i)$ and $J(v,v-k,v-2k+i)$ are easily seen to be isomorphic (see \cite[p.9]{godroy}).  To simplify the exposition, we will often refer to the following global definition.

\begin{definition}\label{global} Fix any nonnegative integers $v>k>i$ and let $X=J(v,k,i)$ denote the corresponding generalized Johnson graph. Assume that $v \geq 2k$ and $(v,k,i) \not= (2k,k,0)$. For convenience, we abbreviate $\Delta=v-2k+2i$. 
\end{definition}

\medskip

Our assumptions imply that $\Delta>0$, and $\Delta>1$ unless $(v,k,i)=(2k+1,k,0)$.

\newpage

\centerline{\newsection{Girth}} 

Recall that the {\em girth} $g(X)$ of a graph $X$ is the length of the shortest cycle in $X$. In this section we derive an expression for the girth of a generalized Johnson graph.  We begin with a lemma that characterizes when two vertices have a common neighbor.

\begin{lemma}\label{dist2-lem} {\bf (Common Neighbor Condition)}
With reference to Definition \ref{global}, let $A$ and $B$ be any vertices and let $x=|A\cap B|$.  Then $A$ and $B$ have a common neighbor if and only if $x\geq \max\{k-\Delta,2i-k\}$.
\end{lemma}

\begin{proof}
Note $A$,$B$ have a common neighbor $C$ if and only if there exists $s \in \integers$, such that every region in Figure~\ref{vennfig} has nonnegative size. 

\begin{figure}[h]
{\centerline{\scalebox{.4}{\includegraphics{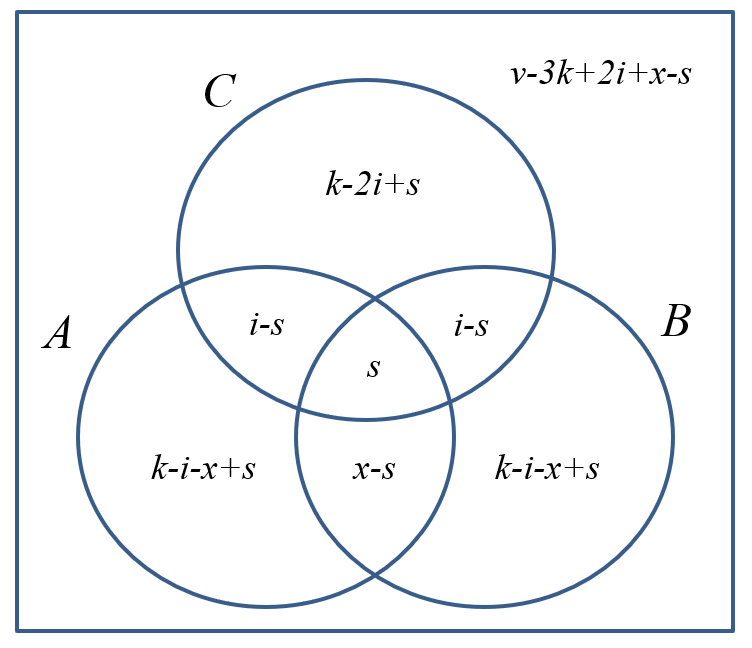}}}}
\caption{Diagram for Lemma~\ref{dist2-lem}}
\label{vennfig}
\end{figure}

\medskip

\noindent
By simplifying the resulting inequalities, we find that $A$ and $B$ have a common neighbor if and only if there exists $s\in\mathbb{Z}$, such that
$$\max\{0,\; i+x-k,\; 2i-k\}\leq s \leq \min \{x,\; i,\; v-3k+2i +x\}.$$
Such an integer $s$ exists if and only if the expression on the left side above does not exceed the expression on the right side.
Under our global assumptions, this is equivalent to $x\geq \max\{k-\Delta,2i-k\}$.
\hfill \end{proof}

\medskip

The lemma above leads immediately to a condition for girth 3.

\begin{lemma} \label{girth3-lem} {\bf (Girth 3)}
With reference to Definition \ref{global}, the girth $g(X)=3$ if and only if $v\geq 3(k-i)$.
\end{lemma}

\begin{proof}
The graph $X$ contains a $3$-cycle iff there exist adjacent vertices $A$ and $B$ that have a common neighbor.  By Lemma \ref{dist2-lem}, this occurs iff $i\geq \max\{k-\Delta,2i-k\}$.  Since $i\geq 2i-k$ holds in all $J(v,k,i)$ graphs, this condition is equivalent to  $v\geq 3(k-i)$. \hfill 
\end{proof}

\medskip

A sufficient condition for the girth to be at most $4$ is the existence of a $4$-cycle.

\begin{lemma} \label{girth4-lem} {\bf (Girth 4)}
With reference to Definition \ref{global}, if $(v,k,i)\not= (2k+1,k,0)$, then $g(X) \leq 4$.
\end{lemma}

\begin{proof}We proceed in three cases.

\textit{Case 1:} $i \geq 2$ or $v > 2k+1$.
In this case, we have $v \geq 2k - i + 2$.  So we can find disjoint sets, $A_1,A_2,A_3,A_4$, and $B_1,B_2,$ and $C$ such that $|A_1|=|A_2|=|A_3|=|A_4|=1$, and $|B_1|=|B_2|=k-i-1$, and $|C|=i$.  Then, as in Figure~\ref{vennfig2}, the following is a 4-cycle in $X$:
$$A_1\cup B_1\cup C,\;\;\; A_2 \cup B_2\cup C,\;\;\; A_3\cup B_1\cup C,\;\;\; A_4\cup B_2\cup C.$$  

\begin{figure}[h]
{\centerline{\scalebox{.4}{\includegraphics{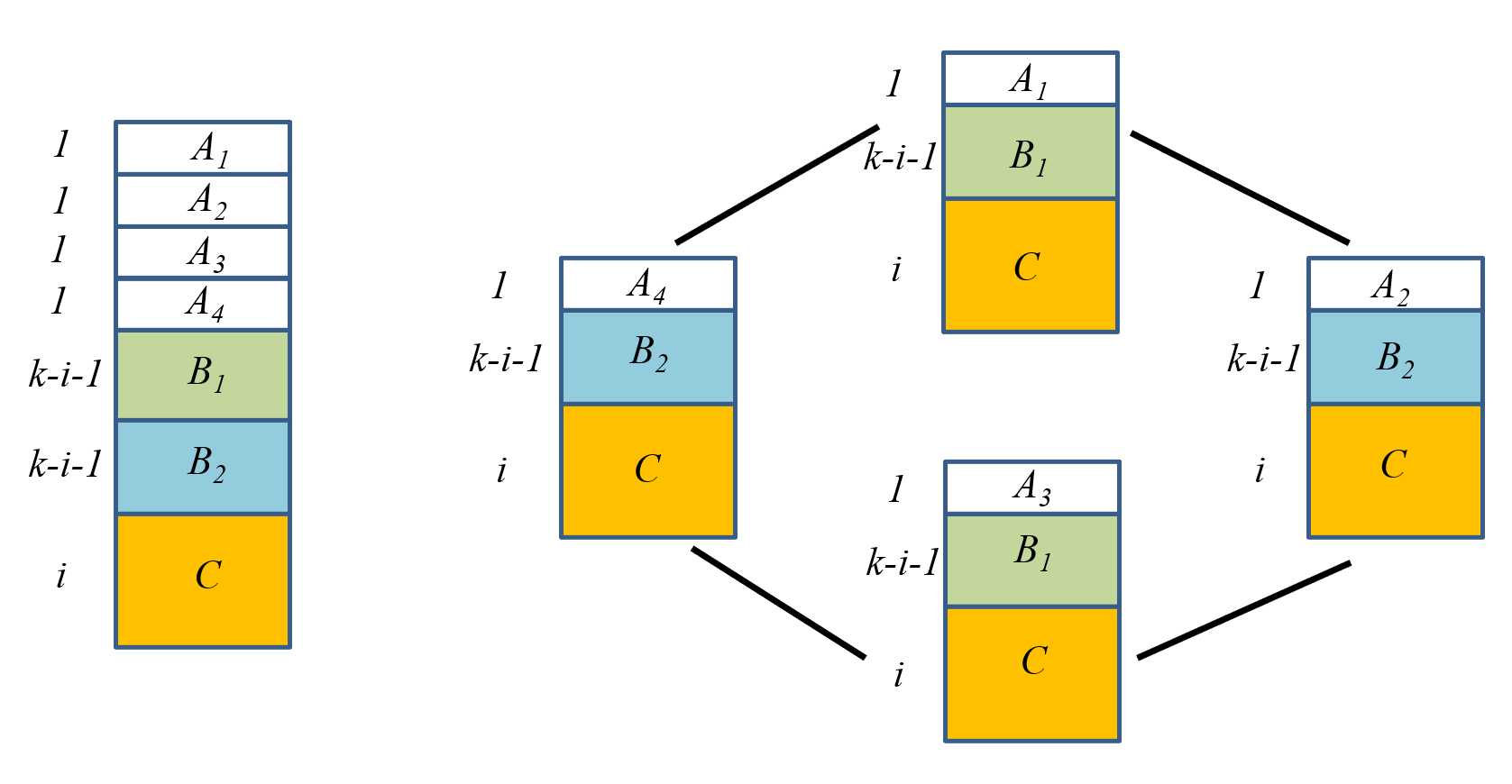}}}}
\caption{Diagram for Case 1 of Lemma~\ref{girth4-lem}}
\label{vennfig2}
\end{figure}

\textit{Case 2:} $i=1$.
In this case, since $v \geq 2k$, we can find disjoint sets $A_1,A_2,A_3,A_4$ and $B_1,B_2$ such that $|A_1|=|A_2|=|A_3|=|A_4|=1$ and $|B_1|=|B_2|=k-2$. Then, as illustrated in Figure~\ref{vennfig3},  the following is a 4-cycle in $X$: $$A_1\cup A_2\cup B_1,\;\;\; A_2\cup A_3\cup B_2,\;\;\;  A_3\cup A_4\cup B_1,\;\;\; A_4\cup A_1\cup B_2.$$  

\begin{figure}[h]
{\centerline{\scalebox{.45}{\includegraphics{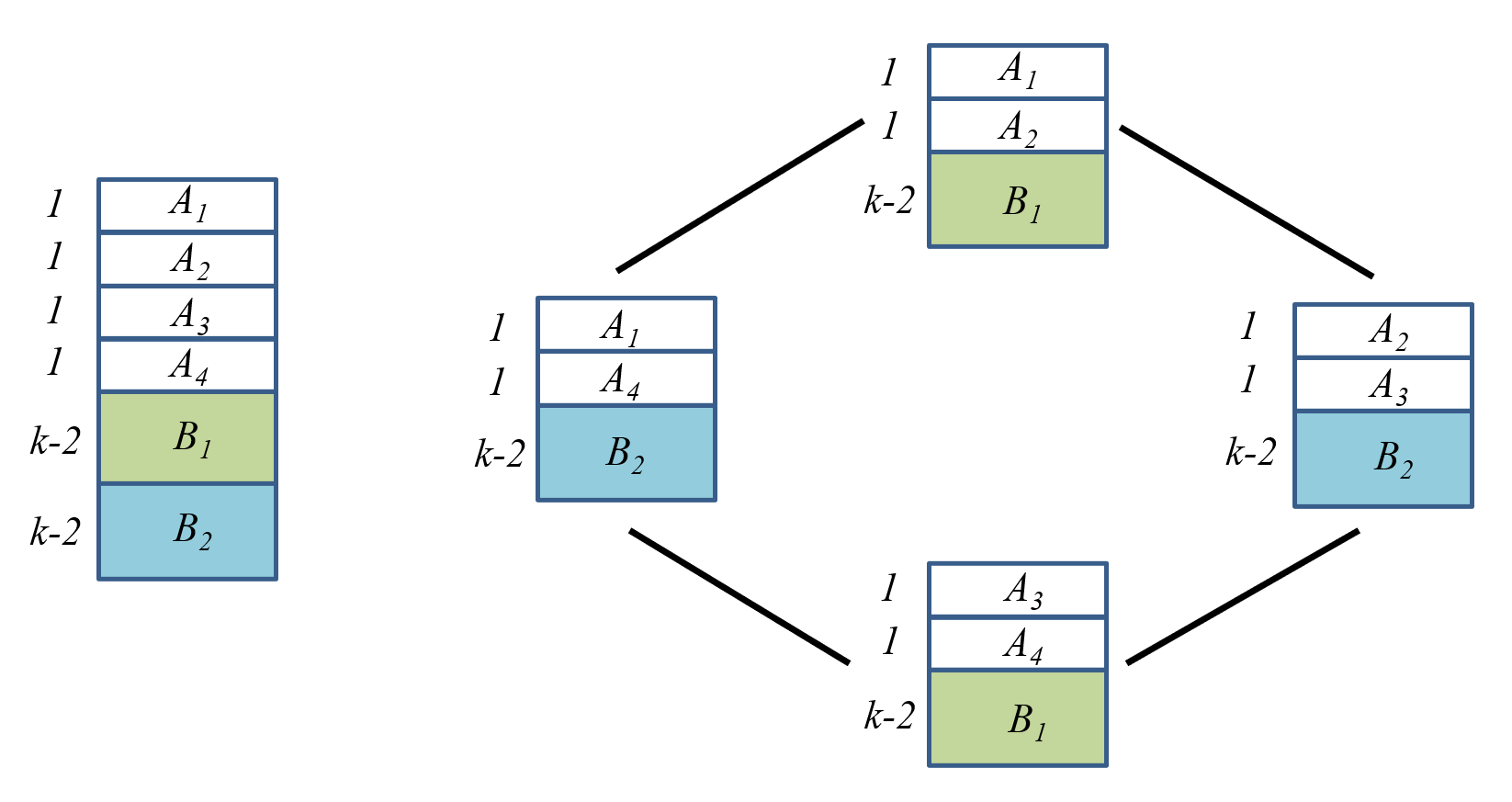}}}}
\caption{Diagram for Case 2 of Lemma~\ref{girth4-lem}}
\label{vennfig3}
\end{figure}

\textit{Case 3:} $i =0$ and $v \leq 2k+1$.
Recall that $v \geq 2k$. So either $v=2k$ or $2k+1$, giving $(v,k,i)=(2k,k,0)$ or $(2k+1,k,0)$. Both are excluded by our assumptions. \hfill 
\end{proof}

\medskip

Combining the above lemmas, we obtain a general expression for the girth.

\begin{theorem}\label{girth-thm}
With reference to Definition \ref{global}, the girth of $X$ is given by
$$g(X)= \left\{
\begin{array}{ll}
3 & \hbox{if \, } v\geq 3(k-i); \\
4 & \hbox{if \, } v<3(k-i) \, \mbox{ and } \, (v,k,i)\not= (2k+1,k,0); \\
5 & \hbox{if \, } (v,k,i)=(5,2,0); \\
6 & \hbox{if \,} (v,k,i)=(2k+1,k,0) \, \mbox{ and } \, k>2.
\end{array}
\right.$$
\end{theorem}

\begin{proof}
The first two cases follow from Lemmas \ref{girth3-lem} and \ref{girth4-lem}.  The remaining cases are odd graphs, for which the girth is well-known. (See, for example, \cite[p.58]{biggs}.) \hfill
\end{proof}

\newpage

\centerline{\newsection{Distance}} 

In this section, we consider questions of distance. With reference to Definition \ref{global}, we derive a general expression for the distance between two vertices in terms of their intersection.  

We begin with the distance between vertices whose intersection is less than $i$.

\begin{lemma}\label{dist3-lem}
With reference to Definition \ref{global}, let $A$ and $B$ be vertices and let $x=|A\cap B|$.  Suppose $x<i$.  Then
$$ {\rm dist}(A,B)  = \left\{
\begin{array}{cl}
3 & \; \hbox{if \, } x < k- \Delta; \\
{\displaystyle \left\lceil \frac{k-x}{k-i} \right\rceil}
& \; \hbox{if \, } x \geq k- \Delta.
\end{array}
\right.$$
\end{lemma}

\begin{proof}We proceed in three cases.

\textit{Case 1:} $x < k- \Delta$.
By Lemma \ref{dist2-lem} and $x<i$, dist$(A,B)\geq 3$.   As in Figure~\ref{vennfig4}, let $A'\subseteq A\setminus B$, such that $|A'|=i-x$.  Let $B' \subseteq B\setminus A$, such that $|B'|=k-i$.  Let $C = A' \cup (A\cap B) \cup B'$.  Then $|C|=k$, and $|A\cap C|= i$, so $C$ is a vertex adjacent to $A$.  Note that $|B\cap C|=k-i+x\geq k-\Delta$.  Also, since $x < k-\Delta$, we have $2i-k<-(v-2k)-x\leq 0$, so $|B \cap C|\geq 2i-k$.  By Lemma \ref{dist2-lem}, dist$(B,C)\leq 2$ and dist$(A,B)=3$ as desired.

\begin{figure}[h]
{\centerline{\scalebox{.45}{\includegraphics{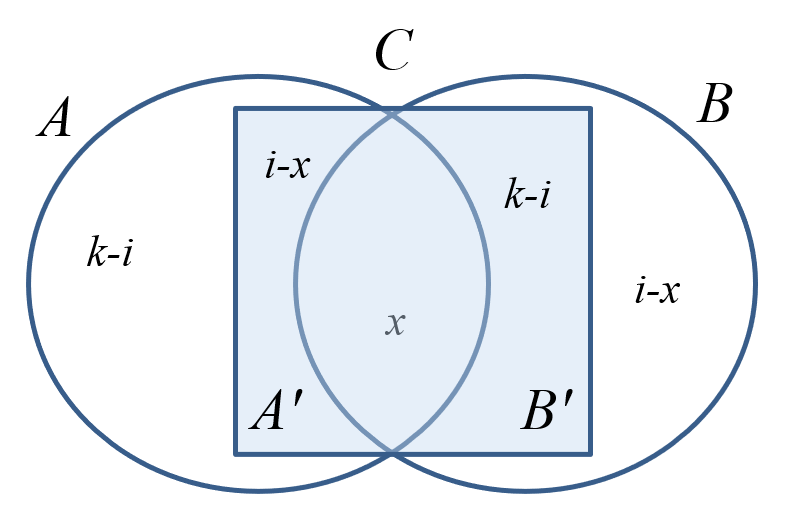}}}}
\caption{Diagram for Case 1 of Lemma~\ref{dist3-lem}}
\label{vennfig4}
\end{figure}

\textit{Case 2:} $x \geq k- \Delta$ and $x \geq 2i-k$.
Since $x<i$, Lemma \ref{dist2-lem} implies dist$(A,B) = 2$.  Also note that $k-i<k-x\leq 2(k-i)$, so $\lceil \frac{k-x}{k-i} \rceil = 2$, as desired.

\textit{Case 3:} $x \geq k- \Delta$ and $x < 2i-k$.
In this case, $k-x>2(k-i)$.  So there exist positive integers $q,m$ such that $k-x = (q+1)(k-i)+m$ with $0<m\leq k-i$.  Let $C = A\cap B$.  Then we can write $A$ and $B$ as disjoint unions
$$A=A_1\cup\cdots\cup A_{q+2}\cup C \hspace{1em} \mbox{and} \hspace{1em} B=B_1\cup\cdots\cup B_{q+2}\cup C,$$
where $|A_j|=|B_j|=k-i$ for $j\in\{1,\ldots,q+1\}$ and $|A_{q+2}|=|B_{q+2}|=m$.  Define
$$X_j = (B_1\cup\cdots\cup B_j)\cup (A_{j+1}\cup\cdots\cup A_{q+2})\cup C$$
 for each $j\in \{1,\ldots,q\}$.  As in Figure~\ref{vennfig5}, $A,X_1,\ldots,X_q$ is a path of length $q$.  Note that $|X_q\cap B|=x + q(k-i) = i-m$. So now, by our case assumptions, 
 $$i > |X_q\cap B| = i-m \geq 2i-k > x \geq k- \Delta.$$ By Lemma~\ref{dist2-lem}, we see that dist$(X_q,B)=2$. So dist$(A,B)\leq q+2 = \lceil \frac{k-x}{k-i} \rceil$.  On the other hand, since adjacent vertices differ only by $k-i$ elements, dist$(A,B) \geq \lceil \frac{k-x}{k-i} \rceil$. \hfill
\end{proof}

\begin{figure}[h]
{\centerline{\scalebox{.45}{\includegraphics{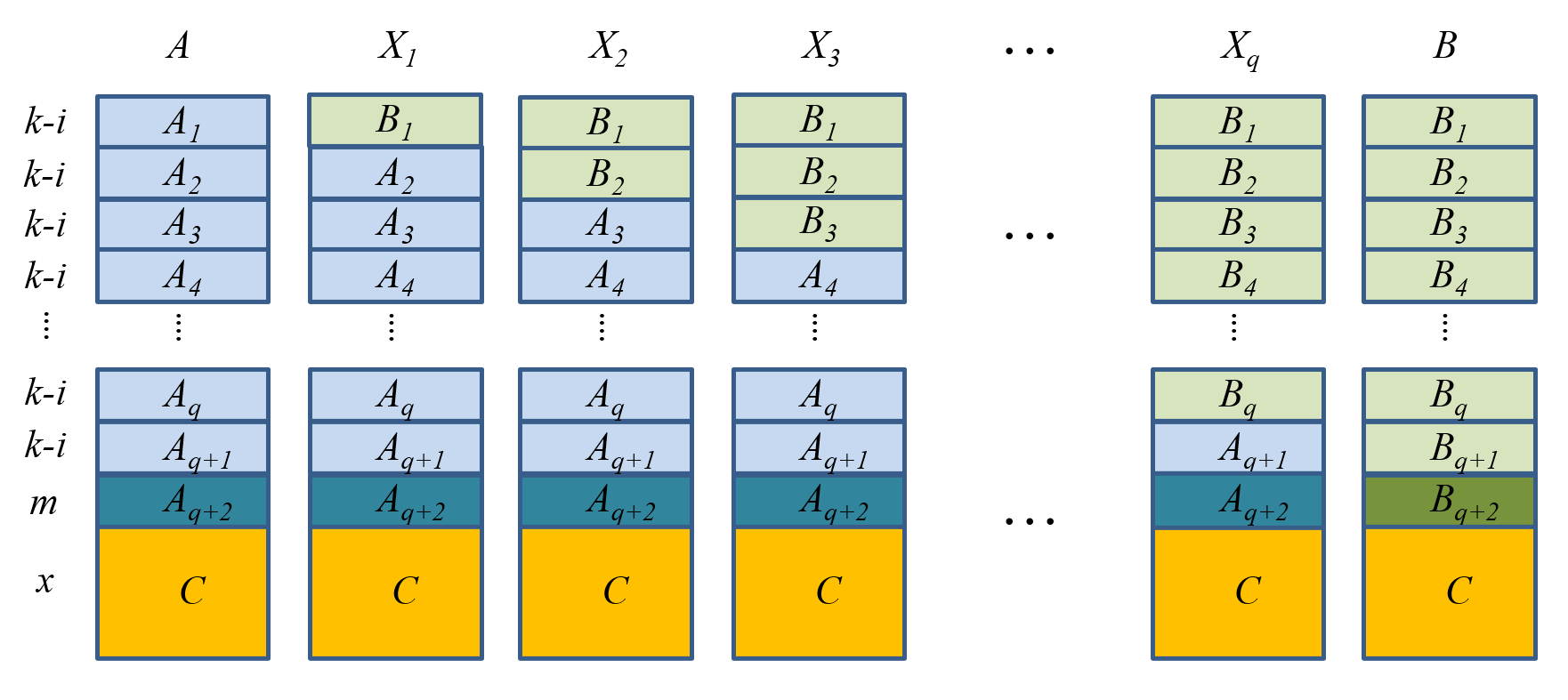}}}}
\caption{Diagram for Case 3 of Lemma~\ref{dist3-lem}}
\label{vennfig5}
\end{figure}

\medskip

It remains to consider the case when $|A \cap B|$ is greater than $i$.  Before we do so, the following result adapts Lemmas 1,2 in \cite{chen3} to generalized Johnson graphs.

\begin{lemma}\label{dist-x-big-lem}
With reference to Definition \ref{global}, let $A$ and $B$ be vertices and let $x=|A\cap B|$.  Suppose there is an $AB$-path of length $d$.
\begin{enumerate}
\item If $d=2p$, then $p\geq \left\lceil \frac{k-x}{\Delta} \right\rceil.$
\item If $d=2p+1$, then $p\geq \left\lceil \frac{x-i}{\Delta} \right\rceil.$
\end{enumerate}
\end{lemma}

\begin{proof}
We argue by induction on $d$. If $d = 0$, then $A=B$ so, $x=k$ and $p=0\geq\lceil \frac{k-x}{\Delta} \rceil$.  If $d=1$, then $x = i$, so $p=0\geq \lceil\frac{x-i}{\Delta}\rceil$.  If $d=2$, then by Lemma \ref{dist2-lem}, $x\geq k-\Delta$.  Hence, $p=1\geq \lceil \frac{k-x}{\Delta}\rceil$.  
Now assume $d\geq 3$ and that the claim holds for all vertices joined by paths of length less than $d$. We proceed in two cases.

\textit{Case 1:} $d = 2p$.
Along the $AB$-path, there exists a vertex $C$ with an $AC$-path of length $2(p-1)$ and a $BC$-path of length 2.  By the inductive hypothesis, $k-|A\cap C|\leq(p-1)\Delta$ and $k-|C\cap B|\leq \Delta$.  Therefore, $k-x=|A\setminus B|\leq |A\setminus C|+|C\setminus B| \leq p\Delta$.  Hence $p \geq \lceil \frac{k-x}{\Delta} \rceil$.

\textit{Case 2:} $d = 2p+1$.
Along the $AB$-path, there exists a vertex $C$ that is adjacent to $B$ and has an $AC$-path of length $2p$. By the inductive hypothesis, $|A\setminus C|\leq p\Delta$.  Therefore, $x-i=|A\cap B|-i\leq |A\setminus C| + |B\cap C| - i\leq p\Delta$.  Hence $p \geq \lceil \frac{x-i}{\Delta} \rceil$. \hfill
\end{proof}

\medskip

The previous lemma implies a lower bound on the distance.  The next result will show that this bound is sharp.

\begin{lemma}\label{dist-path}
With reference to Definition \ref{global}, let $A$ and $B$ be vertices and let $x=|A\cap B|$.  Suppose $x>i$.  Then 
\begin{equation}\label{disty} {\textstyle
{\rm dist}(A,B)= \min\left\{2\left\lceil \frac{k-x}{\Delta} \right\rceil,2\left\lceil \frac{x-i}{\Delta} \right\rceil +1\right\}.}
\end{equation}
\end{lemma}

\begin{proof}
When $x = k$ the result is trivial, so assume $x<k$.  Let $C=A\cap B$ and $D = \overline{A\cup B}$; it follows that $|C|=x$ and $|D|=v-2k+x$.  There exist non-negative integers $q,m$ such that $k-x = q \Delta + m$, with $0< m \leq \Delta$.  We can write A and B as disjoint unions $A=C\cup\{a_1,\ldots,a_{k-x}\}$ and $B=C\cup\{b_1,\ldots,b_{k-x}\}$.  

If $q=0$, then $k-x=m\leq \Delta$, which implies $x \geq -v+3k-2i$.  
The right side of (\ref{disty}) equals 2. Since $x>i$, we also have $x>2i-k$.  By Lemma~\ref{dist2-lem}, dist$(A,B)=2$ as desired.  

Now, assume $q\geq 1$.
For each $j\in\{1,\ldots,q\}$, let
$$A_j=\{a_1,\ldots,a_{(j-1)\Delta +i}\} \hspace{1em} \mbox{and} \hspace{1em} A_j'=\{a_{j\Delta+1},\ldots,a_{k-x}\},$$
$$B_j=\{b_1,\ldots,b_{j\Delta}\} \hspace{1em} \mbox{and} \hspace{1em} B_j'=\{b_{j\Delta -i+1},\ldots,b_{k-x}\},$$
and define
$$X_{2j-1} = D\cup A_j \cup B_j' \hspace{1em} \mbox{and} \hspace{1em} X_{2j} = C\cup B_j \cup A_j'.$$

\begin{figure}[h]
{\centerline{\scalebox{.4}{\includegraphics{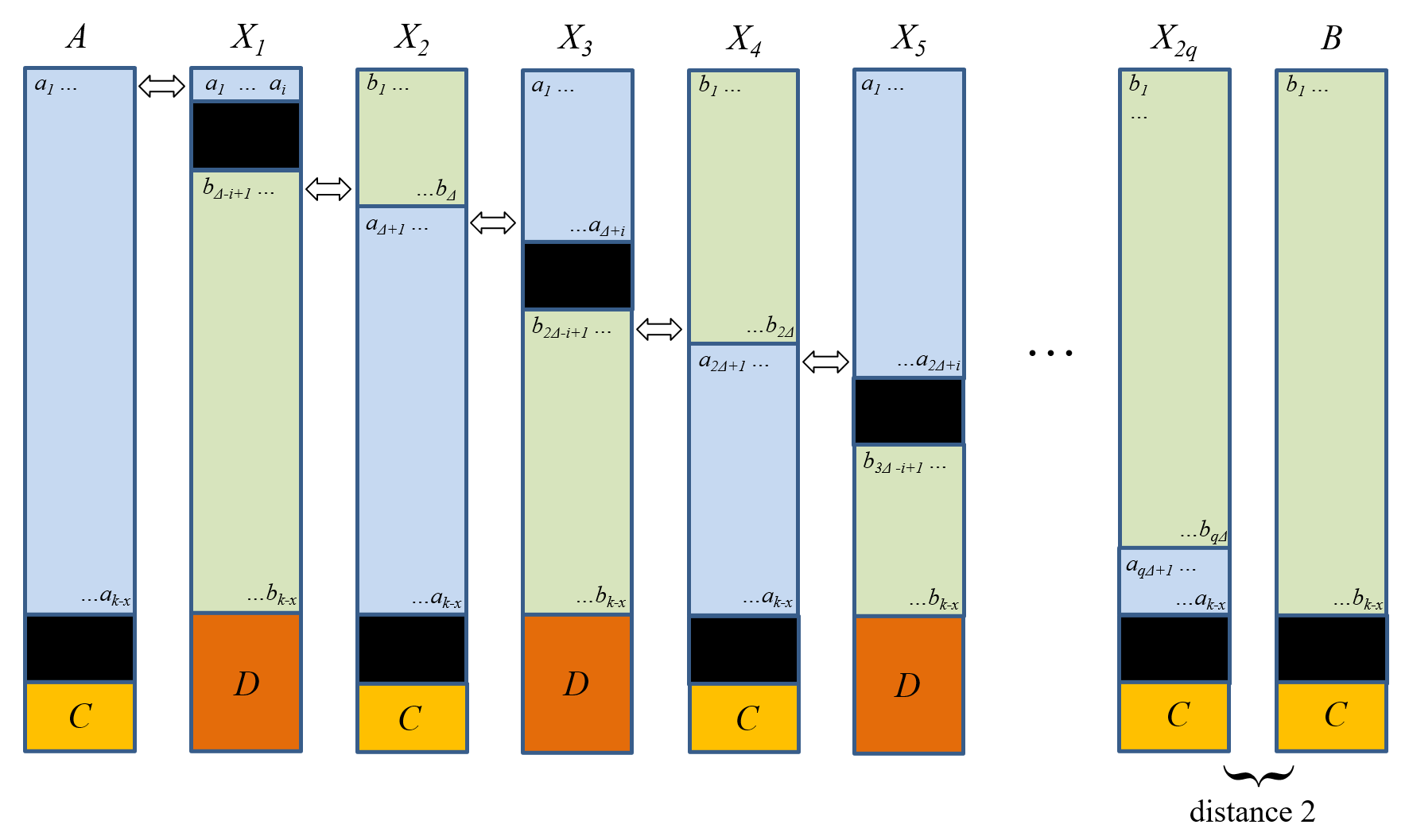}}}}
\caption{Diagram for Lemma~\ref{dist-path}}
\label{vennfig6}
\end{figure}

Then, as in Figure~\ref{vennfig6}, we see that $A,X_1,\ldots,X_{2q}$ is a path of length $2q$.  Note that $|X_{2q} \cap B| = k-m \geq k-\Delta=-v+3k-2i$.  Also, since $m \leq k-x$, we have $|X_{2q}\cap B| \geq x>i\geq 2i-k$.  So dist$(X_{2q},B)=2$  by Lemma \ref{dist2-lem}, and dist$(A,B) \leq 2q+2= 2\lceil \frac{k-x}{\Delta}  \rceil$.

Now, choose any $D' \subseteq D$, $C'\subseteq C$ satisfying $|D'| = |C'|= x-i$.  Let $A' = (B\setminus C')\cup D'$.  Then $A'$ is a vertex adjacent to $A$.  Further, $|A'\cap B|= k-x+i>i$.  By applying the previous argument to $A'$ and $B$, we have dist$(A',B) \leq 2\lceil \frac{k-(k-x+i)}{\Delta} \rceil=2\lceil \frac{x-i}{\Delta} \rceil$.  Therefore dist$(A,B) \leq 2\lceil \frac{x-i}{\Delta} \rceil+1$. 

By Lemma \ref{dist-x-big-lem}, it follows that dist$(A,B)= \min\{2\lceil \frac{k-x}{\Delta} \rceil,2\lceil \frac{x-i}{\Delta} \rceil +1\}$. \hfill 
\end{proof}

\medskip 

From the above results, we obtain a general formula for the distance between two vertices.

\begin{theorem}\label{dist-thm}
With reference to Definition \ref{global}, let $A$ and $B$ be vertices and let $x=|A\cap B|$.  Then
\begin{equation}\label{dist-eq}
\text{dist}(A,B)= \left\{
\begin{array}{ll}
3 & \hbox{if \, } x<\min\{i,k-\Delta\}; \\
\lceil\frac{k-x}{k-i}\rceil & \hbox{if \, } k-\Delta\leq x< i; \\
\min\{2\lceil \frac{k-x}{\Delta} \rceil,2\lceil \frac{x-i}{\Delta} \rceil +1\} & \hbox{if \, } x \geq i. \\
\end{array}
\right.
\end{equation}
\end{theorem}

\begin{proof}
Apply Lemmas \ref{dist3-lem} and \ref{dist-path}.  Note that when $x=i$, we have dist$(A,B)=1=\min\{2\lceil \frac{k-x}{\Delta} \rceil,2\lceil \frac{x-i}{\Delta} \rceil +1\}$. \hfill 
\end{proof}

\newpage

\centerline{\newsection{Diameter}}

In this section, we find the diameter of generalized Johnson graphs.
We begin with a well-known observation about the classical Johnson graphs $J(v,k,k-1)$ that is an easy corollary of Theorem~\ref{dist-thm}.

\begin{corollary}\label{johnson}
With reference to Definition \ref{global}, assume $k=i+1$. Then
$$ \mbox{diam}(X) = k.$$
\end{corollary}

\begin{proof} Fix any vertices $A,B$ and let $x=|A \cap B|$.
Since $v \geq 2k$, we have $k-\Delta \leq 1-i$. So $\min \{i, k-\Delta \} \leq 0$ and the first case in (\ref{dist-eq}) is impossible. If $x \geq i$, then $x=k$ or $k-1$, so $\mbox{dist}(A,B)=0$ or 1. And if $k-\Delta \leq x < i$ then the second case in (\ref{dist-eq}) says $\mbox{dist}(A,B)= k-x$ which ranges between 2 and $k$ as $x$ ranges between 0 and $i-1$. \hfill
\end{proof}

\medskip

The next result determines the maximum value of the expression in Lemma \ref{dist-path} when we exclude the classical Johnson graphs.

\begin{lemma}\label{max-lem}
With reference to Definition \ref{global},
assume $k>i+1$. Let 
$${\textstyle 
f(x) = \min\left\{2\left\lceil \frac{k-x}{\Delta} \right\rceil,2\left\lceil \frac{x-i}{\Delta} \right\rceil +1\right\},
}$$  
and let $\mathcal{I}=\{i+1,\ldots,k\}$. Then 
$$
\max_{x\in\mathcal{I}} \, f(x) =
{\textstyle \left\lceil \frac{k-i-1}{\Delta}\right\rceil+1.}
$$   
\end{lemma}

\begin{proof}
Let $x\in\mathcal{I}$. There exist integers $\epsilon \in \{0,1\}$ and $q,m \geq 0$ such that $k-i-1 = (2q+\epsilon)\Delta+m$ and $0<m\leq \Delta$.  We prove $\max_{x\in\mathcal{I}}f(x)=2q+\epsilon+2$.

Let $x_0=(q+\epsilon)\Delta + i$.  If $x > x_0$, then $2\lceil \frac{k-x}{\Delta} \rceil\leq 2\lceil \frac{k-(x_0+1)}{\Delta} \rceil = 2(q+1) \leq 2q+\epsilon + 2$.  If $x \leq x_0$, then $2\lceil \frac{x-i}{\Delta} \rceil +1 \leq 2\lceil \frac{x_0-i}{\Delta} \rceil +1 = 2(q+\epsilon)+1\leq 2q + \epsilon + 2$.  
Hence, $f(x)\leq 2q+\epsilon+2$.

Let $x_1 = q\Delta + i + 1 + \epsilon(m-1)\in\mathcal{I}$.  It follows that $\lceil \frac{k-x_1}{\Delta}\rceil= q+\epsilon+1$ and $\lceil\frac{x_1-i}{\Delta}\rceil=q+1$.   Therefore, $f(x_1)=\min\{2(q+\epsilon+1),2q+3\}=2q+\epsilon+2$, and the result follows. 
\hfill \end{proof}

\medskip

We can apply the previous lemma to obtain the following corollary, which gives the well-known diameter (see \cite{valencia}) of the Kneser graphs $J(v,k,0)$.

\begin{corollary}\label{kneser} \cite[Theorem 1]{valencia}
With reference to Definition \ref{global}, assume $i=0$. Then
$$ \mbox{diam}(X) =
{\textstyle \left\lceil \frac{k-1}{v-2k}\right\rceil+1.}
$$   
\end{corollary}

\begin{proof} Fix any vertices $A,B$. Let $x=|A \cap B|$ and apply case $x \geq i$ in (\ref{dist-eq}). Since $i=0$, note that $\Delta = v-2k$. If $k=1$, then $X$ is complete, so $\mbox{diam}(X)=1$, as desired.  If $k >1$, Lemma \ref{max-lem} says the maximum value of the distance function is $\left\lceil \frac{k-1}{v-2k}\right\rceil+1,$ as desired. \hfill
\end{proof}

\medskip

We now present the general expression for the diameter. This theorem extends and corrects previous claims found elsewhere (for example, \cite{chen}).

\begin{theorem}
With reference to Definition \ref{global}, we have
$$ \text{diam}(X)= \left\{
\begin{array}{ll}
\lceil \frac{k-i-1}{\Delta} \rceil+1 & \hbox{if \, } v<3(k-i)-1 \text{ or } i=0; \\
3 & \hbox{if \, } 3(k-i)-1\leq v <3k-2i \text{ and } i\neq 0;\\
\lceil\frac{k}{k-i}\rceil & \hbox{if \, } v \geq 3k-2i \text{ and } i\neq 0. \\
\end{array}
\right.$$
\end{theorem}

\begin{proof} By Corollaries \ref{johnson} and \ref{kneser}, the result holds when $i=0$ or $i=k-1$, so assume $0<i<k-1$. We proceed in three cases.

\textit{Case 1}: $v<3(k-i)-1$. 
In this case, $\lceil \frac{k-i-1}{\Delta} \rceil+1 \geq 3$.  Also, $3(k-i) > v \geq 2k$, so 
$\lceil \frac{k-i-1}{\Delta} \rceil+1 \geq 3 > \lceil\frac{k}{k-i}\rceil$.  
By Lemma \ref{max-lem} and Theorem \ref{dist-thm}, there exist vertices $A,B$ such that dist$(A,B)=\lceil \frac{k-i-1}{\Delta} \rceil+1$.  It follows that diam$(X)=\lceil \frac{k-i-1}{\Delta} \rceil+1$, as desired.

\textit{Case 2}: $3(k-i)-1\leq v < 3k-2i$.
 In this case, $\lceil \frac{k-i-1}{\Delta} \rceil+1 \leq 2$. Also, $3k-2i > v \geq 2k$, so  
 $\lceil\frac{k}{k-i}\rceil \leq 2$.  By Theorem \ref{dist-thm}, if $A,B$ are any disjoint vertices, dist$(A,B)=3$, since $\min \{i, k-\Delta\} >0$. Hence, diam$(X)=3$, as desired.

\textit{Case 3}: $v\geq 3k-2i$.
In this case, $\lceil \frac{k-i-1}{\Delta} \rceil+1 \leq 2$.  Also $k \leq \Delta$, so the first case in Theorem \ref{dist-thm} does not occur.  Since $i> 0$, we have $\lceil\frac{k}{k-i}\rceil \geq 2$.  By Theorem \ref{dist-thm}, if $A,B$ are any disjoint vertices, dist$(A,B)=\lceil\frac{k}{k-i}\rceil$.  Hence diam$(X)=\lceil\frac{k}{k-i}\rceil$, as desired. \hfill
\end{proof}

\newpage

\centerline{\newsection{Odd Girth}}

With reference to Definition \ref{global}, Theorem~\ref{girth-thm} says that the girth of $X$ satisfies $3 \leq g(X) \leq 6$.  When $g(X) = 3$ or $5$, the odd girth equals the girth, so $og(X) = 3$ or $5$, respectively. We consider the cases of even girth separately below. We will see that, in all cases, the odd girth is given by the following expression:
\begin{equation}\label{OddG}
    og(X) = 2\left\lceil\frac{k-i}{\Delta}\right\rceil +1.
\end{equation} 
 
\noindent
{\bf Girth 6 case.}

\medskip


With reference to Definition \ref{global}, assume $g(X)=6$. Theorem~\ref{girth-thm} tells us that $X$ is an odd graph, with $(v,k,i)=(2k+1,k,0)$ for some integer $k>2$. In this case, the odd girth equals $2k+1$, a result originally due to \cite{poljak}. Observe that, since $i=0$ and $\Delta = 1$, this agrees with the expression in (\ref{OddG}). Because this value also follows readily from our previous results, we include a proof here.

\begin{lemma}\cite[Cor. 11]{poljak}\label{poljakLem}
	With reference to Definition \ref{global}, assume $k>2$. Let $X=J(2k+1,k,0)$. Then $$og(X)= 2k+1.$$
\end{lemma}

\begin{proof}
	Suppose $A_1, A_2, ..., A_{2t+1}$ is any odd cycle in $X$. Then dist$(A_1,A_3)=2$, so $|A_1 \cap A_3|=k-1$ by Theorem~\ref{dist-thm}. The cycle also has a path of length $2t-1$ from $A_3$ to $A_1$, so $t \geq k$  by Lemma~\ref{dist-x-big-lem}. It remains to show that a closed walk of length $2k+1$ exists.
 
If $k=2d$ then $v=4d+1$. Let
$$ A= \{1,...,2d\}, \quad B= \{d+1,...,3d\}, \quad C = \{ 2d+1,...,4d\}.$$

If $k=2d+1$ then $v=4d+3$. Let
$$ A= \{1,...,2d+1\}, \quad B= \{d+2,...,3d+2\}, \quad C = \{ 2d+3,...,4d+3\}.$$
In both cases, $|A \cap B| = |B \cap C| = d$ and $|A \cap C|=0$. So, by Theorem~\ref{dist-thm}, dist$(A,B)=$ dist$(B,C)=k$, and dist$(A,C)=1$, as desired. \hfill 
\end{proof}

\bigskip

\noindent
{\bf Girth 4 case.}

\medskip 

With reference to Definition \ref{global}, assume $g(X)=4$. It is helpful to recall that $\Delta >1$ unless $(v,k,i)=(2k+1,k,0).$  But $g(X)=4$, so Theorem~\ref{girth-thm} implies that 
\begin{equation}\label{DeltaBd}
1 < \Delta < k-i.
\end{equation}
To establish (\ref{OddG}), it will help to have conditions to guarantee that a pair of vertices are at distance $\lceil\frac{k-i}{\Delta}\rceil.$

\begin{lemma} \label{g4-lem}
	With reference to Definition \ref{global}, assume $g(X)=4$.  Let $r=\lceil\frac{k-i}{\Delta}\rceil$. Fix any vertices $A$,$B$ and let $x=|A \cap B|$. Then the following hold.

\begin{enumerate} 
\item If $r$ is odd and $x\in \{ \lfloor\frac{k+i-\Delta}{2}\rfloor , \lceil\frac{k+i-\Delta}{2}\rceil \}$, then dist$(A,B)=r$. 
\item If $r$ is even and $x\in \{ \lfloor\frac{k+i}{2}\rfloor , \lceil\frac{k+i}{2}\rceil \}$, then dist$(A,B)=r$.
\end{enumerate}
\end{lemma}

\vspace{-.2cm}

\begin{proof} 
(i).  Assume $r$ is odd and $x\in \{ \lfloor\frac{k+i-\Delta}{2}\rfloor , \lceil\frac{k+i-\Delta}{2}\rceil \}$. Write $r = 2d+1$ so that 
\begin{equation}\label{2dineq}
2d <\frac{k-i}{\Delta}\leq 2d+1.
\end{equation} 

Claim 1: $\left\lceil \frac{x-i}{\Delta}\right\rceil =d$.
%
%
By (\ref{DeltaBd}), we know $\Delta \geq 2$, so (\ref{2dineq}) implies
$2d-1+\frac{2}{\Delta}\le \frac{k-i}{\Delta}\le 2 d+1.$
Subtract $1$ and multiply by $\Delta/2$ to get
$\Delta(d-1)+1\le \frac{k+i-\Delta}{2}-i\le \Delta d.$
Now 
we have
$\Delta(d-1)+1\le x-i\le \Delta d.$
Dividing by $\Delta$ yields
$\left\lceil \frac{x-i}{\Delta}\right\rceil =d$, as claimed.
%
%

\medskip

Claim 2: $x>i$. 
By (\ref{DeltaBd}), we know $\Delta<k-i$, and therefore $r>1$. This fact implies $d>0$. 
By Claim 1, it now follows that $x>i$, as desired. 

\medskip

Claim 3: $\lceil\frac{k-x}{\Delta}\rceil \geq d+1$. 
To see this, note that
 $   \left\lceil\frac{k-x}{\Delta}\right\rceil+\left\lceil\frac{x-i}{\Delta}\right\rceil\geq \frac{(k-x)+(x-i)}{\Delta}=\frac{k-i}{\Delta}>2d,$
where the final inequality is from (\ref{2dineq}). By  
Claim 1, we have $\lceil\frac{k-x}{\Delta}\rceil \geq d+1$, as desired. 

\medskip

Applying Claims 1-3 to Theorem~\ref{dist-thm}, we have dist$(A,B) = 2d+1$, proving (i).

\medskip

\noindent
(ii). Assume $r$ is even and $x\in \{ \lfloor\frac{k+i}{2}\rfloor , \lceil\frac{k+i}{2}\rceil \}$. Write $r = 2d$ so that 
\begin{equation}\label{2dineqB}
2d -1<\frac{k-i}{\Delta}\leq 2d.
\end{equation}  

Claim 1: $\left\lceil \frac{k-x}{\Delta}\right\rceil =d$.
By (\ref{DeltaBd}), we know $\Delta \geq 2$, so (\ref{2dineqB}) implies
$2d-2+\frac{2}{\Delta}\le \frac{k-i}{\Delta}\le 2 d.$
Multiply by $\Delta/2$ and simplify to obtain
$\Delta(d-1)+1\le k - \frac{k+i}{2}\le \Delta d.$
Now we have $\Delta(d-1)+1\le k-x\le \Delta d.$
Dividing by $\Delta$ yields
$\left\lceil \frac{k-x}{\Delta}\right\rceil =d$, as claimed. 

\medskip

Claim 2: $\lceil\frac{x-i}{\Delta}\rceil \geq d$. 
To see this, note that $\left\lceil\frac{k-x}{\Delta}\right\rceil+\left\lceil\frac{x-i}{\Delta}\right\rceil\geq \frac{(k-x)+(x-i)}{\Delta}=\frac{k-i}{\Delta}>2d-1,$
where the final inequality is from (\ref{2dineqB}). By Claim 1, we have $\lceil\frac{x-i}{\Delta}\rceil \geq d$, as desired.

\medskip

Claim 3:  $x>i$. 
By (\ref{DeltaBd}), we know $\Delta<k-i$, and therefore $r>1$. This fact implies $d>0$. 
By Claim 2, it now follows that $x>i$, as desired. 

\medskip

Applying Claims 1-3 to Theorem~\ref{dist-thm}, we have dist$(A,B) = 2d$, proving (ii).
\hfill
\end{proof}

\medskip

\noindent
We now can prove a lower bound for the odd girth.

\begin{lemma} \label{fancy-LB}
	With reference to Definition~\ref{global}, assume $g(X)=4$.   Then $$og(X) \geq 2\left\lceil \frac{k-i}{\Delta} \right\rceil+1.$$
\end{lemma}

\begin{proof} Assume $og(X)= 2r+1$ and fix adjacent vertices $A,B$ on a $(2r+1)$-cycle. Let $C$ be the vertex opposite edge $AB$ on that cycle, so that dist$(A,C) =$ dist$(B,C)=r$.  We will prove $r \geq \frac{k-i}{\Delta}$ in two cases, depending on the parity of $r$. 

\medskip
		
Case $r=2d+1$.   We claim $|A\cap C| \leq\Delta d+i$. If not, let $x = |A\cap C|$ and suppose $x \ge \Delta d+i+1$. Then by Theorem~\ref{dist-thm}, since $r$ is odd, $d=\lceil \frac{x-i}{\Delta}\rceil\ge \lceil \frac{\Delta d + 1}{\Delta}\rceil=d+1$, a contradiction.   Similarly, 
 $|B\cap C| \leq\Delta d+i$.  Therefore, $v-2k+i 
\geq|C\setminus(A\cup B)| 
\geq k-2(\Delta d+i)$. It follows that $\Delta-i\ge k-2(\Delta d+i)$, which implies $r \geq \frac{k-i}{\Delta}$.

\medskip

Case $r=2d$. We claim $|A\cap C| \geq k-\Delta d$. If not, let $x = |A\cap C|$ and suppose $x < k-\Delta d$. Since $r$ is even and $x< k - \Delta$, Theorem~\ref{dist-thm} implies $2d = 2 \lceil \frac{k-x}{\Delta} \rceil$, a contradiction.  Similarly, 
 $|B\cap C| \geq k-\Delta d$. 
  But now $k=|C|\geq|A\cap C|+|B\cap C|-|A\cap B|\geq 2(k-\Delta d)-i$, and therefore $r \geq \frac{k-i}{\Delta}$. \hfill
\end{proof}
 
\begin{lemma} \label{fancy}
	With reference to Definition~\ref{global}, assume $g(X)=4$.   Then $$og(X) = 2\left\lceil \frac{k-i}{\Delta} \right\rceil+1.$$
\end{lemma}
 
\begin{proof} Let $r = \lceil \frac{k-i}{\Delta} \rceil.$ Given Lemma~\ref{fancy-LB}, we must construct a closed walk of length $2r+1$. We proceed in two cases, depending on the parity of $r$.

\medskip

 Case $r=2d+1$.  Let $x=\frac{k+i-\Delta}{2}$.    By Theorem~\ref{girth-thm}, we have $2k\leq v<3(k-i)$, so $0\leq x \leq k-i$.  Fix any adjacent vertices $A,B$.  Choose $A_0\subseteq A\setminus B$ and $B_0\subseteq B\setminus A$ with $|A_0|=\lfloor x \rfloor$ and $|B_0|= \lceil x \rceil$.  Let $C_0=\overline{A\cup B}$, and $C=A_0\cup B_0\cup C_0$.  Then $|C|=\lfloor x \rfloor + \lceil x \rceil +(v-2k+i)=k$, so $C\in V(X)$.  Also, $|A\cap C|=\lfloor x \rfloor $ and $|B\cap C|=\lceil x \rceil$. By Lemma~\ref{g4-lem}(i), dist$(A,C) =$ dist$(B,C)=r$ and $X$ has a closed walk of length $2r+1$.

\medskip

Case $r=2d$.   Let $x=\frac{k+i}{2}$,  and notice that $0\leq x-i \leq k-i$. Fix any adjacent vertices $A,B$.  Choose  $A_0\subseteq A\setminus B$ and $B_0\subseteq B\setminus A$ with $|A_0|=\lfloor x \rfloor -i$ and $|B_0|=\lceil x \rceil-i$.  Let $C_0=A\cap B$, and $C=A_0\cup B_0\cup C_0$.  Then $|C|=k$, so $C\in V(X)$.  Also, $|A\cap C|=\lfloor x \rfloor$ and $|B\cap C|= \lceil x \rceil$. By Lemma~\ref{g4-lem}(ii), dist$(A,C) =$ dist$(B,C)=r$ and $X$ has a closed walk of length $2r+1$.  \hfill
\end{proof}

\bigskip

\noindent
{\bf General case.}
 
\medskip

\begin{theorem} \label{wow-thm}
	With reference to Definition~\ref{global}, the odd girth of $X$ is given by
 \begin{equation}\label{girth-eq}
 og(X)=2\left\lceil\frac{k-i}{\Delta}\right\rceil+1.
 \end{equation}
\end{theorem}

\begin{proof}
 By Theorem~\ref{girth-thm}, the girth of $X$ satisfies $3 \leq g(X) \leq 6$.  
 
If $g(X)=3$, then $og(X)=3$. In this case, $v \geq 3k-3i$ by Theorem~\ref{girth-thm}, so $k-i \leq v-2k+2i$ and equality holds in (\ref{girth-eq}).

If $g(X)=4$, then equality holds in (\ref{girth-eq}) by Lemma~\ref{fancy}.   

If $g(X)=5$, then $og(X)=5$. In this case, $(v,k,i)=(5,2,0)$ by Theorem~\ref{girth-thm}, so equality holds in (\ref{girth-eq}).

Finally, if $g(X)=6$, then $(v,k,i)=(2k+1,k,0)$ for some integer $k>2$ by Theorem~\ref{girth-thm}. But then $og(X)=2k+1$ by Lemma~\ref{poljakLem} and equality holds in (\ref{girth-eq}). \hfill
\end{proof}



\newpage







\nocite{agong2018girth}

\bibliographystyle{plain}
\bibliography{main}
\end{document}